\documentclass[a4paper,12pt]{article}
\usepackage{times, url}
\usepackage{amsmath,amssymb,amsthm,amsfonts}

\newtheorem{theorem}{Theorem}[section]
\newtheorem{lemma}{Lemma}[section]

\newtheorem{remark}{Remark}[section]

\begin{document}
\setcounter{page}{1} 
\vspace{10mm}

\begin{center}
{\LARGE \bf  A Criterion for Deficient Numbers \\
Using the Abundancy Index \\
and Deficiency Functions}
\vspace{8mm}

{\large \bf Jose Arnaldo B. Dris}
\vspace{3mm}

Department of Mathematics and Physics, Far Eastern University \\ 
Nicanor Reyes Street, Sampaloc, Manila, Philippines \\
e-mail: \url{jadris@feu.edu.ph}, \url{josearnaldobdris@gmail.com}
\vspace{2mm}

\end{center}
\vspace{10mm}

\noindent
{\bf Abstract:} We show that $n$ is almost perfect if and only if $I(n) - 1 < D(n) \leq I(n)$, where $I(n)$ is the abundancy index of $n$ and $D(n)$ is the deficiency of $n$. This criterion is then extended to the case of integers $m$ satisfying $D(m)>1$. 
\vspace{10mm}

\section{Introduction}
If $n$ is a positive integer, then we write $\sigma(n)$ for the sum of the divisors of $n$.  A number $n$ is \emph{almost perfect} if $\sigma(n)=2n - 1$.  It is currently unknown whether there are any other almost perfect numbers apart from those of the form $2^k$, where $k \geq 0$.

We denote the abundancy index $I$ of the positive integer $w$ as $I(w) = \displaystyle\frac{\sigma(w)}{w}$.  We also denote the deficiency $D$ of the positive integer $x$ as $D(x) = 2x - \sigma(x)$ \cite{OEIS-A033879}.

\section{Preliminary Lemmata}
We begin with some preliminary results.

Note that $I(y)=D(y)$ if and only if
$$\sigma(y) = 2y^2 - y\sigma(y)$$
which corresponds to
$$(y+1)\sigma(y) = 2y^2 \Longleftrightarrow \sigma(y) = \frac{2y^2}{y+1} = \frac{2y^2 - 2}{y + 1} + \frac{2}{y + 1} = 2(y - 1) + \frac{2}{y + 1}.$$
Since $\sigma(y)$ and $2(y-1)$ are both integers, this implies that $(y+1) \mid 2$, from which it follows that $y+1 \leq 2$. Hence, $y \leq 1$. Together with $1 \leq y$, this means that $y = 1$.

We state this result as our initial lemma.

\begin{lemma}\label{Lemma0}
$I(n)=D(n)$ if and only if $n=1$.
\end{lemma}

Next, we show conditions that are sufficient and necessary for $n$ to be almost perfect.

\begin{lemma}\label{Lemma1}
If $n$ is a positive integer which satisfies the inequality $$\frac{2n}{n + 1} \leq I(n) < 2,$$ then $n$ is almost perfect.
\end{lemma}

\begin{proof}
Let $n$ be a positive integer, and suppose that
$$\frac{2n}{n + 1} \leq I(n) < 2.$$
Then we have
$$\frac{2n^2}{n + 1} \leq \sigma(n) < 2n.$$
But
$$\frac{2n^2}{n + 1} = 2n - 2 + \frac{2}{n + 1} \leq \sigma(n) < 2n.$$
Since $\sigma(n)$ is an integer and $n \geq 1$, this last chain of inequalities forces
$$\sigma(n) = 2n - 1,$$
and we are done.
\end{proof}

\begin{lemma}\label{Lemma2}
If $n$ is almost perfect, then $n$ satisfies the inequality $$\frac{2n}{n + 1} \leq I(n) < 2.$$
\end{lemma}

\begin{proof}
Let $n$ be a positive integer, and suppose that $\sigma(n) = 2n - 1.$

It follows that
$$I(n) = \frac{\sigma(n)}{n} = 2 - \frac{1}{n} < 2.$$

Now we want to show that
$$\frac{2n}{n + 1} \leq I(n).$$

Assume to the contrary that $I(n) < \frac{2n}{n + 1}$.  (Note that this forces $n > 1.$)  Mimicking the proof in Lemma \ref{Lemma1}, we have
$$\sigma(n) < \frac{2n^2}{n + 1} ,$$
from which it follows that
$$\sigma(n) < 2n - 2 + \frac{2}{n + 1} < (2n - 2) + 1 = 2n - 1.$$

This contradicts our assumption that $n$ is almost perfect.  Hence the reverse inequality
$$\frac{2n}{n + 1} \leq I(n)$$
holds.
\end{proof}

\begin{remark}\label{Remark1}
By their definition, all almost perfect numbers are automatically deficient.  But of course, not all deficient numbers are almost perfect.
\end{remark}

By Remark \ref{Remark1}, it seems natural to try to establish an upper bound for the abundancy index of an almost perfect number $n$, that is ${\it strictly}$ less than 2, and which (perhaps) can be expressed as a rational function of $n$ (similar to the form of the lower bound given in Lemma \ref{Lemma1} and Lemma \ref{Lemma2}).

\begin{lemma}\label{Lemma3}
If $n$ is a positive integer which satisfies the inequality $$\frac{2n}{n + 1} \leq I(n) < \frac{2n + 1}{n + 1},$$
then $n$ is almost perfect.
\end{lemma}

\begin{proof}
Let $n$ be a positive integer, and suppose that
$$\frac{2n}{n + 1} \leq I(n) < \frac{2n + 1}{n + 1}.$$

Again, mimicking the proof in Lemma \ref{Lemma1}, we have
$$\frac{2n^2}{n + 1} \leq \sigma(n) < \frac{2n^2 + n}{n + 1},$$
from which it follows that
$$2n - 2 + \frac{2}{n + 1} \leq \sigma(n) < 2n - 1 + \frac{1}{n + 1}.$$

Since $n \geq 1$, this last chain of inequalities forces the equation
$$\sigma(n) = 2n - 1$$
to be true.  Consequently, $n$ is almost perfect, and we are done.
\end{proof}

We now show that the (nontrivial) upper bound obtained for the abundancy index of $n$ in Lemma \ref{Lemma3} is also necessary for $n$ to be almost perfect.

\begin{lemma}\label{Lemma4}
If $n$ is almost perfect, then $n$ satisfies the inequality $$\frac{2n}{n + 1} \leq I(n) < \frac{2n + 1}{n + 1}.$$
\end{lemma}

\begin{proof}
It suffices to prove that if $n$ is almost perfect, then the inequality
$$I(n) < \frac{2n + 1}{n + 1}$$
holds.  To this end, assume to the contrary that
$$\frac{2n + 1}{n + 1} \leq I(n).$$
Mimicking the proof in Lemma \ref{Lemma3}, we obtain
$$2n - 1 + \frac{1}{n + 1} = \frac{2n^2 + n}{n + 1} \leq \sigma(n),$$
from which it follows that
$$2n - 1 < \sigma(n).$$
This contradicts our assumption that $n$ is almost perfect, and we are done.
\end{proof}

\section{Main Results}
Collecting all the results from Lemma \ref{Lemma1}, Lemma \ref{Lemma2}, Lemma \ref{Lemma3} and Lemma \ref{Lemma4}, we now have the following theorem.

\begin{theorem}\label{Theorem1}
Let $n$ be a positive integer.  Then $n$ is almost perfect if and only if the following chain of inequalities hold:
$$\frac{2n}{n + 1} \leq I(n) < \frac{2n + 1}{n + 1}.$$
\end{theorem}

\begin{remark}\label{Remark2}
Note that equality holds in
$$\frac{2n}{n + 1} \leq I(n)$$
if and only if $n=1$.

Also, it is trivial to verify that the known almost perfect numbers $n = 2^k$ (for integers $k \geq 0$) satisfy the inequalities in Theorem \ref{Theorem1}.  In fact, Theorem \ref{Theorem1} can be used to rule out particular families of integers from being almost perfect numbers. (For example, $n_1 = p^r$ and $n_2 = p^r q^s$, where $p, q$ are odd primes and $r, s$ are even, are easily shown not to satisfy $\sigma(n_i) = 2{n_i} - 1$ ($i=1,2$) using the criterion in Theorem \ref{Theorem1}.)

This can, of course, also be attempted for the case $M = {2^r}{b^2}$ with $\sigma(M)=2M-1$, where $r \geq 1$ and $b$ is an odd composite indivisible by $3$,  although a complete proof appears to be difficult \cite{AntalanDris}.
\end{remark}

We now give a proof for the following result (which was originally conjectured in \\
\url{http://arxiv.org/pdf/1308.6767v4.pdf}).

\begin{theorem}\label{Theorem2}
The bounds in Theorem \ref{Theorem1} are best-possible.
\end{theorem}

\begin{proof}
The bounds
$$\frac{2n}{n + 1} \leq I(n) = \frac{\sigma(n)}{n} < \frac{2n + 1}{n + 1}$$
are easily seen to be equivalent to
$${2n}\cdot{n} \leq (n+1)\cdot\sigma(n) < {2n}\cdot{n} + n$$
which further implies that
$$n\left(2n - \sigma(n)\right) \leq \sigma(n)$$
and
$$\sigma(n) - n < \left(2n - \sigma(n)\right)n$$
so that we obtain
$$I(n) - 1 < 2n - \sigma(n) = D(n) \leq I(n).$$
Since
$$D(n) = 1 \leq I(n) = \frac{\sigma(n)}{n} < 2 = D(n) + 1$$
when $n$ is almost perfect, we obtain the claimed result.
\end{proof}

\begin{remark}\label{Remark3}
Following the proof of Theorem \ref{Theorem2}, and by using Lemma \ref{Lemma0}, we have the following additional observations when 
$m > 1$:
{
\begin{enumerate}
	\item{If $D(m) < I(m) < 2$, then $m$ is almost perfect, by Lemma \ref{Lemma1}.}
	\item{The case $I(m) < D(m) < 2$ cannot occur, as it implies that $2m - 2 < \sigma(m) < 2m$, so that $D(m)=2m-\sigma(m)=1$, contradicting $1 \leq I(m) < D(m) = 1$.}
	\item{If $I(m) < 2 \leq D(m)$, then $m$ is not almost perfect, by Lemma \ref{Lemma1}.}
\end{enumerate}
}
\end{remark}

Next, we attempt to extend the results in Theorem \ref{Theorem1} to the case of positive integers $m$ satisfying $D(m)>1$. To begin with, notice that, if we write $1=D(n)$, then the bounds in Theorem \ref{Theorem1} take the following form:
$$\frac{2n}{n + D(n)} \leq I(n) < \frac{2n + D(n)}{n + D(n)}.$$
By Remark \ref{Remark2}, equality holds in
$$\frac{2n}{n + D(n)} \leq I(n)$$
if and only if $n=1$, which is true if and only if $I(n)=D(n)$ by Lemma \ref{Lemma0}.

Now, assume that $m$ is a positive integer with $D(m)>1$.  We want to show that the following theorem holds.

\begin{theorem}\label{Theorem3}
Let $m$ be a positive integer, and suppose that $D(m)>1$. Then we have the following bounds for the abundancy index of $m$, in terms of the deficiency of $m$:
$$\frac{2m}{m + D(m)} < I(m) < \frac{2m + D(m)}{m + D(m)}.$$
\end{theorem}

\begin{proof}
Assume that $m$ is a positive integer satisfying $D(m)>1$.  (In particular, note that $m>1$ (by Lemma \ref{Lemma0}), and therefore that $I(m)>1$ (since $I(m)=1$ if and only if $m=1$).)

Suppose to the contrary that
$$I(m) \leq \frac{2m}{m + D(m)} = \frac{2m}{3m - \sigma(m)}.$$
Then we have (noting that $3m - \sigma(m) > 2m - \sigma(m) > 1$)
$$3m\sigma(m) - \left(\sigma(m)\right)^2 \leq 2m^2.$$
Dividing through by $m^2$, we get
$$\left(I(m)\right)^2 - 3I(m) + 2 \geq 0$$
which implies that
$$\left(I(m) - 2\right)\left(I(m) - 1\right) \geq 0.$$
This is a contradiction, as we know that $1 < I(m) < 2$.

Now, assume that
$$I(m) \geq \frac{2m + D(m)}{m + D(m)} = \frac{4m - \sigma(m)}{3m - \sigma(m)}.$$
Then we obtain (noting that $4m - \sigma(m) > 3m - \sigma(m) > 2m - \sigma(m) > 1$)
$$3m\sigma(m) - \left(\sigma(m)\right)^2 \geq 4m^2 - m\sigma(m).$$ 
Again, dividing through by $m^2$, we get
$$\left(I(m)\right)^2 - 4I(m) + 4 \leq 0$$
which implies that
$$\left(I(m) - 2\right)^2 \leq 0.$$
This contradicts $1 < I(m) < 2$.

Consequently, we have the bounds
$$\frac{2m}{m + D(m)} < I(m) < \frac{2m + D(m)}{m + D(m)}$$
if $D(m)>1$, and we are done.
\end{proof}

We end this section with the following result, which closely parallels that of Theorem \ref{Theorem2}.

\begin{theorem}\label{Theorem4}
The bounds in Theorem \ref{Theorem3} are best-possible.
\end{theorem}

\begin{proof}
The bounds
$$\frac{2m}{m + D(m)} < I(m) < \frac{2m + D(m)}{m + D(m)}$$
are easily seen to be equivalent to
$$2m < \sigma(m) + D(m)I(m) < 2m + D(m)$$
which further implies that
$$D(m) =  2m - \sigma(m) < D(m)I(m) < \left(2m - \sigma(m)\right) + D(m) = 2D(m)$$
so that we obtain
$$1 < I(m) < 2$$
since $D(m)>1$.
Since $1 < I(m) < 2$ also holds for deficient integers $m$ satisfying $D(m)>1$, the claimed result follows.
\end{proof}

\section{Conclusion}
The results in this article originated from the author's attempts to show that $d$ is almost perfect, if $e^f d^2$ is an odd perfect number with Euler prime $e$ satisfying $e \equiv f \equiv 1 \pmod 4$.  (That is, since 
$$I(d) < I(d^2) = \frac{2}{I(e^f)} \leq \frac{2e}{e+1},$$
the author was hoping to show $d<e$, and thereby prove the Descartes-Frenicle-Sorli conjecture (i.e., $f=1$) as an immediate consequence.)

It is now known that, in fact, one has
$$I(d^2) < \frac{2d^2}{d^2 + 1}$$
and
$$I(d) < \frac{2d}{d+1},$$
so that $d^2$ (and therefore, $d$) are not almost perfect.

Additionally, Brown \cite{Brown} has recently announced a proof for $e < d$, and a partial proof that $e^f < d$ holds ``in many cases".

Nonetheless, work is in progress in \cite{AntalanDris} to try to rule out even almost perfect numbers other than the powers of two.

\section{Acknowledgments}
The author expresses his gratitude to Will Jagy for having answered his question at \\
\url{http://math.stackexchange.com/q/462307} from which the proofs of Lemmas \ref{Lemma1}, \ref{Lemma2}, \ref{Lemma3} and \ref{Lemma4} were patterned after. The author would also like to thank John Rafael M. Antalan from the Central Luzon State University, for helpful chat exchanges on this topic.

Lastly, the author thanks Keneth Adrian P. Dagal for motivating him to pursue the completion of this paper.

\bibliographystyle{amsplain}

\end{document}